\documentclass[11pt]{article}

\usepackage{array}
\usepackage{multirow}

\makeatletter
\newcommand{\thickhline}{%
    \noalign {\ifnum 0=`}\fi \hrule height 1pt
    \futurelet \reserved@a \@xhline
}
\newcolumntype{"}{@{\hskip\tabcolsep\vrule width 1pt\hskip\tabcolsep}}
\makeatother

\usepackage{graphics}
\usepackage{epsfig}
\usepackage{subfigure}
\usepackage{amssymb}
\usepackage{amsthm}
\usepackage{mathrsfs}
\usepackage{hhline}
\usepackage{longtable}
\usepackage{amsmath}
\usepackage{array}
\usepackage{float}
\usepackage{picinpar}%for window
\usepackage{multirow}
\usepackage{cite}
\usepackage{enumerate}
\usepackage{xcolor}
\usepackage[mathlines]{lineno}
\modulolinenumbers[2]
\newcommand*\patchAmsMathEnvironmentForLineno[1]{%
\expandafter\let\csname old#1\expandafter\endcsname\csname #1\endcsname  \expandafter\let\csname oldend#1\expandafter\endcsname\csname end#1\endcsname  \renewenvironment{#1}%
{\linenomath\csname old#1\endcsname}%
{\csname oldend#1\endcsname\endlinenomath}}%
\newcommand*\patchBothAmsMathEnvironmentsForLineno[1]{%
\patchAmsMathEnvironmentForLineno{#1}%
\patchAmsMathEnvironmentForLineno{#1*}}%
\AtBeginDocument{%
\patchBothAmsMathEnvironmentsForLineno{equation}%
\patchBothAmsMathEnvironmentsForLineno{align}%
\patchBothAmsMathEnvironmentsForLineno{flalign}%
\patchBothAmsMathEnvironmentsForLineno{alignat}%
\patchBothAmsMathEnvironmentsForLineno{gather}%
\patchBothAmsMathEnvironmentsForLineno{multline}%
}

\newtheorem{theorem}{Theorem}[section]
\newtheorem{lemma}[theorem]{Lemma}%[section]
\newtheorem{example}{Example}[section]
\newtheorem{corollary}[theorem]{Corollary}%[section]
\newtheorem{problem}{Problem}[section]

\newtheorem{conjecture}{Conjecture}[section]

\numberwithin{equation}{section}

\newtheorem{defi}{Definition}[section]

\usepackage{graphicx}
\graphicspath{%
    {converted_graphics/}% inserted by PCTeX
    {/}% inserted by PCTeX
}

\begin{document}
\baselineskip18truept
\normalsize
\begin{center}
{\mathversion{bold}\Large \bf Affirmative Solutions On Local Antimagic Chromatic Number}

\bigskip
{\large  Gee-Choon Lau{$^{a,}$}\footnote{Corresponding author.}, Ho-Kuen Ng{$^b$}, Wai-Chee Shiu{$^c$}}\\

\medskip

\emph{{$^a$}Faculty of Computer \& Mathematical Sciences,}\\
\emph{Universiti Teknologi MARA (Segamat Campus),}\\
\emph{85000, Johor, Malaysia.}\\
\emph{geeclau@yahoo.com}\\

\medskip

\emph{{$^b$}Department of Mathematics, San Jos\'{e} State University,}\\
\emph{San Jos\'e CA 95192 USA.}\\
\emph{ho-kuen.ng@sjsu.edu}\\

\medskip

%{\large Wai-Chee Shiu\footnote{E-mail: wcshiu@math.hkbu.edu.hk}}\\
\emph{{$^c,$\footnote{The third author was supported by Tianjin Research Program of Application Foundation and Advanced Technology (No.14JCYBJC43100) for visiting Tianjin University of Technology and Education.}}Department of Mathematics, The Chinese University of Hong Kong, Hong Kong, P.R. China,}\\
\emph{College of Global Talents, Beijing Institute of Technology, Zhuhai, P.R. China.}\\
% Department of Mathematics, The Chinese University of Hong Kong, Hong Kong, P.R. China;\\
 %           College of Global Talents, Beijing Institute of Technology, Zhuhai, P.R. China\\
\emph{wcshiu@associate.hkbu.edu.hk}\\

\end{center}

%{\blue Shiu}\quad {\red Lau}\quad {\violet Shiu 2}\\ {\magenta Shiu or Question} \quad {\cyan Maybe no use}

%\maketitle

\begin{abstract}
An edge labeling of a connected graph $G = (V, E)$ is said to be local antimagic if it is a bijection $f:E \to\{1,\ldots ,|E|\}$ such that for any pair of adjacent vertices $x$ and $y$, $f^+(x)\not= f^+(y)$, where the induced vertex label $f^+(x)= \sum f(e)$, with $e$ ranging over all the edges incident to $x$.  The local antimagic chromatic number of $G$, denoted by $\chi_{la}(G)$, is the minimum number of distinct induced vertex labels over all local antimagic labelings of $G$. In this paper, we give counterexamples to the lower bound of $\chi_{la}(G \vee O_2)$ that was obtained in [Local antimagic vertex coloring of a graph, Graphs and Combin., 33 : 275 - 285 (2017)]. A sharp lower bound of $\chi_{la}(G\vee O_n)$ and sufficient conditions for the given lower bound to be attained are obtained. Moreover, we settled Theorem 2.15 and solved Problem 3.3 in the affirmative. We also completely determined the local antimagic chromatic number of complete bipartite graphs.

% Include keywords, PACS and mathematical subject classification numbers as needed.

\end{abstract}

\section{Introduction}\label{intro}
A connected graph $G = (V, E)$ is said to be local antimagic if it admits a local antimagic edge labeling, i.e., a bijection $f : E \to \{1,\ldots ,|E|\}$ such that the induced vertex labeling  $f^+ : V \to \mathbb Z$ given by $f^+(x) = \sum f(e)$ (with $e$ ranging over all the edges incident to $x$) has the property that any two adjacent vertices have distinct induced vertex labels.  The number of distinct induced vertex labels under $f$ is denoted by $c(f)$, and is called the {\it color number} of $f$. The {\it local antimagic chromatic number} of $G$, denoted by $\chi_{la}(G)$, is $\min\{c(f) : f\mbox{ is a local antimagic labeling of } G\}$. In~\cite{Haslegrave}, Haslegrave proved that the local antimagic chromatic number is well-defined for every connected graph other than $K_2$. Thus, for every connected graph $G\not=K_2$, $\chi_{la}(G)\ge \chi(G)$, the chromatic number of $G$. %that every connected graph $G$ admits a local antimagic edge labeling except $G=K_2$. \\

For any graph $G$, the graph $H=G \vee O_n$, $n\ge 1$,  is defined by  $V(H)=V(G)\cup\{v_i: 1\le i\le n\}$ and  $E(H)=E(G)\cup\{uv_i : u \in V(G)\}$. In~\cite[Theorem 2.16]{Arumugam}, it was claimed that for any $G$ with order $m\ge4$, $$\chi_{la}(G) + 1 \le \chi_{la}(G\vee O_2) \le \begin{cases}
\chi_{la}(G)+1 & \mbox{if } m \mbox{ is even,} \\
\chi_{la}(G) + 2 & \mbox{if } m \mbox{ is odd}.
\end{cases}$$

In Section 2, we give counterexamples to the above lower bound for each $m\ge 3$. A sharp lower bound is then given. Moreover, sufficient conditions for the above lower bound to be attained are also presented. In Section 3, we settled~\cite[Theorem~2.15]{Arumugam} and solved  \cite[Problem~3.3]{Arumugam} in the affirmative. In Section 4, we completely determined the local antimagic chromatic number of complete bipartite graphs.

\section{Counterexamples and sharp bound}

In this section, we will make use of the existence of magic rectangles. From \cite{Chai, Reyes}, we know that a $h\times k$ magic rectangle exists when $h,k\ge 2$, $h\equiv k\pmod{2}$ and $(h,k)\ne (2,2)$. For $a,b\in \mathbb{Z}$ and $a\le b$, we use $[a,b]$ to denote the set of integers from $a$ to $b$. We first introduce some notation about matrices.\\

Let $m,n$ be two positive integers. For convenience, we use $M_{m,n}$ to denote the set of $m\times n$ matrices over $\mathbb{Z}$. For any matrix $M\in M_{m,n}$, let $r_i(M)$ and $c_j(M)$ denote the $i$-th row sum and the $j$-th column sum of $M$, respectively.\\

We shall assign the integers in $[1, q+r+qr]$ to matrices $PR\in M_{1,r}$, $QR\in M_{q,r}$ and $QP=(PQ)^T\in M_{q,1}$ such that the matrix
\begin{equation*}
M=\begin{pmatrix}
* & PR\\
QP & QR\end{pmatrix}\end{equation*}
has the following properties:
\begin{enumerate}[P.1]
\item Each integer in $[1, q+r+qr]$ appears once.
\item $r_{i+1}(M)$ is a constant not equal to $r_1(M)+c_1(M)$, $1\le i\le q$.
\item $c_{j+1}(M)$ is a constant not equal to $r_{i+1}(M)$ or $r_1(M)+c_1(M)$, $1\le j\le r$.
\end{enumerate}

Let $\{u\}$, $\{v_1, v_2, \ldots, v_q\}$ and $\{w_1, w_2, \ldots, w_r\}$ be the three independent vertex set of $K(1,q,r)$, $r\ge q\ge 2$. For $1\le  i\le q, 1\le j\le r$, let the $i$-entry of $QP$ be the edge label of $uv_i$, the $j$-entry of $PR$ be the edge label of $uw_j$, the $(i,j)$-entry of $QR$ be the edge label of $v_iw_j$. It follows that $r_1(M) + c_1(M)$ is the sum of all the incident edge labels of $u$, $r_{i+1}(M)$ is the sum of all the incident edge labels of $v_i$, and $c_{j+1}(M)$ is the sum of all the incident edge labels of $w_j$. Thus, $M$ corresponds to a local antimagic labeling of $K(1,q,r)$ with $\chi_{la}(K(1,q,r))=3$ if such $M$ exists.

\begin{theorem}\label{thm-K12r} For $r\ge 2$, $\chi_{la}(K(1,2,r))=3$. \end{theorem}

\begin{proof} Suppose $r$ is even. If $r=2$, a required labeling is given by
\begin{equation*}
M=\begin{pmatrix}
* & 1 & 5\\
6 & 7 & 2\\
8 & 3 & 4
\end{pmatrix}
\end{equation*}

Consider $r\ge 4$. Let $A$ be a $3\times (r+1)$ magic rectangle. Exchanging columns and exchanging rows if necessary so that $3(r+1)$ is put at the $(1,1)$-entry of $A$. Now, $M$ is obtained by letting $PR$ be the $1\times r$ matrix obtained from the first row of $A$ by deleting the $(1,1)$-entry; letting $QP$ be the $3\times 1$ matrix obtained from the first column of $A$ by deleting the $(1,1)$-entry; and letting $QR$ be the $2\times r$ matrix obtained from $A$ by deleting the first row and the first column.

It is easy to check that $c_1(M)+r_1(M)=\frac{(r+4)(3r+4)}{2}-6(r+1)$ $\not=$ $r_2(M)=r_3(M)=\frac{(r+1)(3r+4)}{2}$ $\not=$ $c_{j+1}=\frac{3(3r+4)}{2}$, $1\le j\le r$.

Suppose $r$ is odd. If $r=3$, a required labeling is given by
\begin{equation*}
M=\begin{pmatrix}
* & 2 & 4 & 6\\
8 & 5 & 11 & 3\\
10 & 9 & 1 & 7
\end{pmatrix}
\end{equation*}

Consider $r\ge 5$. Now for $r \equiv 1 \pmod{4}$, let $r = 4s + 1, s \ge 1$. The entries of a required labeling matrix $M$ is given in tabular form as follows.

%%%%%%%%%%%%%
%\vskip-0.5cm
\begin{table}[H]
\footnotesize
$PR=$
      \centering
\setlength{\extrarowheight}{1pt}
\setlength{\tabcolsep}{2.5pt}
\begin{tabular}{|c|c|c|c|c|c"c"c|c|c|c|c|c|}
\hline
$1$  & $3$ & $\cdots$  & $4s-5$ &  $4s-3$ & $4s-1$ & $10s+3$ &  $2$ & $4$ &  $\cdots$ & $4s-4$ & $4s-2$ & $4s$ \\
\hline
\end{tabular}
     \centering
\end{table}
\vskip-0.5cm
 \begin{table}[H]
\footnotesize
$QP+QR=$
      \centering
\setlength{\extrarowheight}{1pt}
\setlength{\tabcolsep}{0.8pt}
\begin{tabular}{|c"c|c|c|c|c|c"c"}
\hline
$12s+5$ & $12s+3$ & $6s+1$ & $\cdots$ &  $4s+5$  & $10s+5$ & $4s+3$ & $4s+1$  \\
\hline
$12s+4$ & $6s+2$ & $12s+2$ & $\cdots$ & $10s+6$ & $4s+4$ & $10s+4$ & $4s+2$   \\
\hline
\end{tabular}
%     \centering
\end{table}
\vskip-0.5cm
 \begin{table}[H]
\footnotesize
%$QP+QR=$
      \centering
\setlength{\extrarowheight}{1pt}
\setlength{\tabcolsep}{0.8pt}
\begin{tabular}{"c|c|c|c|c|c|}
\hline
 $10s+2$ & $8s+1$ & $\cdots$ &  $6s+5$ & $8s+4$ & $6s+3$ \\
\hline
 $8s+2$ & $10s+1$ & $\cdots$ & $8s+5$ & $6s+4$ & $8s+3$ \\
\hline
\end{tabular}
%     \centering
\end{table}

%%%%%%%%%%%%%

 Clearly, we get $c_1(M)+r_1(M)=8s^2+36s+12$ $\not=$ $r_2(M)=r_3(M)=32s^2+27s+6$ $\not=$ $c_{j+1}(M)=18s+6$, $1\le j\le r$.\\ \\ %Each vertex in $A, B$, and $C$ has label $$, $$, and $$ respectively.\\

%%%%%%%%%%%%%%%%%%%%%%%%%%%%%%%%%%%%%%%%%%%%

Finally for $r \equiv 3 \pmod{4}$, let $r = 4s + 3, s \ge 1$. To get $PR$, we assign $2k-1$ to column $k$ if $1\le k \le 2s+2$, and assign $2k-4s-4$ to column $k$ if $2s+3\le k\le 4s+3$. For row 1 of $QR$, we assign $6s+6-k$ to column $k$ if $k=1,3,5,\ldots, 2s+1$; assign $12s+10-k$ to column $k$ if $k=2,4,6,\ldots,2s+2$; assign $10s+9-k$ to column $k$ if $k=2s+3, 2s+5, 2s+7, \ldots, 4s+3$; and assign $12s+10-k$ to column $k$ if $k=2s+4, 2s+6, 2s+8, \ldots, 4s+2$. For row 2 of $QR$, we assign $12s+10-k$ to column $k$ if $k=1,3,5,\ldots,2s+1$; assign $6s+6-k$ to column $k$ if $k=2,4,6,\ldots,2s+2$; assign $12s+10-k$ to column $k$ if $k=2s+3, 2s+5, 2s+7, \ldots, 4s+3$; and assign $10s+9-k$ to column $k$ if $k=2s+4, 2s+6, 2s+8, \ldots, 4s+2$. For $QP$, the two entries are $12s+10$ and $12s+11$. Lastly, we exchange the labels $4s-1$ and $4s+6$; the labels $4s-2$ and $6s+8$; and the labels $4s+2$ and $8s+7$. The resulting matrices are given by the following tables.
%\vskip-0.5cm
\begin{table}[H]
\footnotesize
$PR=$
      \centering
\setlength{\extrarowheight}{1pt}
\setlength{\tabcolsep}{2.pt}
\begin{tabular}{|c|c|c|c|c|c|c|c"c|c|c|c|c|c|c|}
\hline
1& 3& 5& $\cdots$ & $4s-3$ & \mathversion{bold}${4s+6}$ & $4s+1$ & $4s+3$ &  2 & 4 & $\cdots$ & $4s-4$ & \mathversion{bold} $6s+8$ & $4s$ & \mathversion{bold} $8s+7$\\
\hline
\end{tabular}
     \centering
\end{table}
\vskip-0.5cm

 \begin{table}[H]
\hskip2cm\footnotesize
$QP+QR=$
 %     \centering
\setlength{\extrarowheight}{1pt}
\setlength{\tabcolsep}{0.5pt}
\begin{tabular}{|c"c|c|c|c|c|c|c|c"}
\hline
$12s+10$ & $6s+5$ & $12s+8$ & $6s+3$ & $\cdots$ & $4s+7$ & $10s+10$ & $4s+5$ & $10s+8$   \\
\hline
$12s+11$ & $12s+9$ & $6s+4$ & $12s+7$ & $\cdots$ & $10+11$ & \mathversion{bold} $4s-1$ & $10s+9$ & $4s+4$   \\
\hline
\end{tabular}
 %    \centering
\end{table}
\vskip-0.5cm%\hskip1cm
 \begin{table}[H]
\footnotesize
%$QP+QR=$
    \centering
\setlength{\extrarowheight}{1pt}
\setlength{\tabcolsep}{0.6pt}
\begin{tabular}{"c|c|c|c|c|c|c|}
\hline
$8s+6$ & $10s+6$ &  $\cdots$ & $8s+10$ & \mathversion{bold} $4s-2$ & $8s+8$ & $6s+6$ \\
\hline
 $10s+7$ & $8s+5$  & $\cdots$ & $6s+9$ & $8s+9$ & $6s+7$ & \mathversion{bold}$4s+2$ \\
\hline
\end{tabular}
   \centering
\end{table}
\vskip-0.5cm
%According to the above tables, we know vertex $u$ has label $8s^2+38s+27$ and vertex $w_j, 1\le j\le r$ has label $18s+15$. However, vertex $v_1$ and $v_2$ have labels $y_1=32s^2+61s+29$ and $y_2=32s^2+63s+31$ respectively.  To make $y_1 = y_2$, we perform the following exchanges.  Note that none of these exchanges would modify the vertex labels of $w_j$. %Only the value of sum of entries in $PR$ and in $QP$  would be changed.

%\begin{enumerate}[(a)]
 % \item Interchange the labels $4s-2$ and $6s+8$.  The value of $y_1$ is decreased by $2s+10$.
 % \item Interchange the labels $4s-1$ and $4s+6$.  The value of $y_2$ is decreased by 7.
 % \item Interchange the labels $4s+2$ and $8s+7$.  The value of $y_2$ is decreased by $4s+5$.
%\end{enumerate} In total, the value of $y_1$ is decreased by $2s+10$, and the value of $y_2$ is decreased by $4s+12$.

Thus, we now have a required $M$ with $c_1(M)+r_1(M)=8s^2 + 44s + 49$ $\not=$ $r_2(M) = r_3(M)=32s^2+59s+19$ $\not=$ $c_{j+1}(M)=18s+15, 1\le j\le r$.   \end{proof}

Observe that $K(1,2,r) = K(1,r) \vee O_2$. Obviously, $\chi_{la}(K(1,r)) = r+1, r\ge 2$.

\begin{corollary}\label{cor-ctreg} For each $m\ge 3$, there exists a graph $G$ of order $m$ such that $\chi_{la}(G\vee O_2)  - \chi_{la}(G) = 3-m \le 0$. \end{corollary}

Corollary~\ref{cor-ctreg} serves as counterexamples to the lower bound of~\cite[Theorem 2.16]{Arumugam}. Interested readers may refer to~\cite{LauNgShiu-CM} for more general results on $\chi_{la}(K(p,q,r)), r\ge q\ge p\ge 1$. The next theorem gives a sharp lower bound of $\chi_{la}(G\vee O_n)$ for $n\ge 1$.

\begin{theorem} For $n\ge 1$, $\chi_{la}(G\vee O_n)\ge \chi(G) + 1$ and the bound is sharp.
\end{theorem}

\begin{proof} It is obvious that for $n\ge 1$, we have $\chi_{la}(G\vee O_n)\ge \chi(G\vee O_n)= \chi(G) + 1$. In~\cite{LauNgShiu-chila}, the authors obtained that for $h\ge 2, k\ge 1$, $\chi_{la}(C_{2h}\vee O_{2k}) = 3$ and $\chi_{la}(C_{2h-1}\vee O_{2k-1}) = 4$. Since $\chi(C_{2h}) = 2$ and $\chi(C_{2h-1}) = 3$, the bound is sharp.  \end{proof}

Observe that if $\chi_{la}(G)=\chi(G)$, then $\chi_{la}(G\vee O_n)\ge \chi(G\vee O_n) = \chi(G) + 1 = \chi_{la}(G) + 1$. Thus we have proved the sufficiency of the following conjecture.

\begin{conjecture}\label{conj-joint} For $n\ge 1$, $\chi_{la}(G\vee O_n)\ge\chi_{la}(G)+1$ if and only if $\chi_{la}(G)=\chi(G)$.  \end{conjecture}

In~\cite{Arumugam}, we have for $m\ge2$, $\chi(C_{2m-1}) = \chi_{la}(C_{2m-1}) = 3 = \chi_{la}(C_{2m})$ and $\chi(C_{2m})=2$. This provides a supporting evidence that the conjecture holds. \\

Let $a_{i,j}$ be the $(i,j)$-entry of a magic $(m,n)$-rectangle with row constant $n(mn+1)/2$ and column constant $m(mn+1)/2$. The following theorems partially answer Conjecture~\ref{conj-joint}.

\begin{theorem}\label{thm-G+On} Suppose $G$ is of order $m\ge 3$ with $m\equiv n\pmod{2}$ and $\chi(G)=\chi_{la}(G)$. If (i) $n\ge m$,  or (ii) $m \ge n^2/2$ and $n \ge 4$, then $\chi_{la}(G\vee O_n)= \chi_{la}(G) + 1$. \end{theorem}

\begin{proof} Let $G$ has size $e$ such that $V(G)=\{u_i : 1\le  i\le m\}$ and $V(O_n) = \{v_j : 1\le j\le n\}$. Suppose $f$ is a local antimagic labeling of $G$ that induces a $t$-coloring of $G$.\\

Define $g: E(G\vee O_n) \to [1,e+mn]$  by
\begin{align*}
g(uv) & = f(uv)\mbox{ for each $uv\in E(G)$},\\
g(u_iv_j) & = e + a_{i,j}\mbox{ for $1\le i\le m, 1\le j\le n$}.
\end{align*}

It is clear that $g$ is a bijection such that
\begin{align*}
g^+(u_i) & = f^+(u_i) + ne + n(mn+1)/2\mbox{ for } 1\le i\le m,\\
g^+(v_j) & = me + m(mn+1)/2 \mbox{ for } 1\le j\le n.
\end{align*}

(i) Suppose $n\ge m$, we have $g^+(u_i) > g^+(v_j)$ for all $i$ and $j$.\\

(ii) Suppose $m \ge n^2/2$ and $n \ge 4$. We proceed to show that $g^+(v_j) - g^+(u_i) = (m - n)e + (m - n)(mn + 1)/2 - f^+(u_i) > 0$ for all $i$ and $j$. Note that $e \le m(m - 1)/2$, and $f^+(u_i) \le e + (e - 1) + \cdots + (e - m + 2) = (2e - m + 2)(m - 1)/2$.  Thus, $2(g^+(v_j) - g^+(u_i))\ge 2(m - n)e + (m - n)(mn + 1) - (2e - m + 2)(m - 1) = 2(1 - n)e + (m - n)(mn + 1) + (m - 2)(m - 1) \ge (1 - n)m(m - 1) + (m - n)(mn + 1) + (m - 2)(m - 1) = 2m^2 + mn - 3m - mn^2 - n + 2 > m(2m - n^2) + m(n - 4) + (m - n) \ge 0$. \\

In either case, $g$ is a local antimagic labeling that induces a $(t+1)$-coloring of $G\vee O_n$. Hence, $\chi_{la}(G\vee O_n)\le \chi_{la}(G)+1$. Since $\chi_{la}(G\vee O_n)\ge \chi(G\vee O_n) = \chi(G)+1 = \chi_{la}(G)+1$, the theorem holds. \end{proof}

 Hence, we may assume that $m>n$.

\begin{theorem}\label{thm-Gr+On} Suppose $G$ is an $r$-regular graph of order $m\ge 3$ with $m\equiv n\pmod{2}$ and $\chi(G)=\chi_{la}(G)$.  If $m> n$ and $r\ge \frac{(m-n)(mn+1)}{2mn}$, then $\chi_{la}(G\vee O_n)= \chi_{la}(G) + 1$.
 \end{theorem}

\begin{proof} Let $V(G)=\{u_i : 1\le  i\le m\}$ and $V(O_n) = \{v_j : 1\le j\le n\}$. Note that the size of $G$ is $mr/2$. Suppose $f$ is a local antimagic labeling of $G$ that induces a $t$-coloring of $G$. Define $g: E(G\vee O_n) \to [1, mr/2+mn]$ by
\begin{align*}
g(uv) & = f(uv)+mn\mbox{ for each $uv\in E(G)$},\\
g(u_iv_j) & = a_{i,j}\mbox{ for $1\le i\le m, 1\le j\le n$}.
\end{align*}

It is clear that $g$ is a bijection such that
\begin{align*}
g^+(u_i) & = f^+(u_i)+ mnr  + n(mn+1)/2 \mbox{ for }1\le i\le m,\\
g^+(v_j) & = m(mn+1)/2, \mbox{ for }1\le j\le n.
\end{align*}

 Now, $2(g^+(u_i)-g^+(v_j))=2f^+(u_i)+2mnr+(n-m)(mn+1)$. Since $r\ge \frac{(m-n)(mn+1)}{2mn}$, we have $g^+(u_i) > g^+(v_j)$. This means $g$ is a local antimagic labeling that induces a $(t+1)$-coloring of $G\vee O_n$. Hence, $\chi_{la}(G\vee O_n)\le \chi_{la}(G)+1$. Since $\chi_{la}(G\vee O_n)\ge \chi(G\vee O_n) = \chi(G)+1 = \chi_{la}(G)+1$, the theorem holds. 
\end{proof}

From~\cite[Theorem 2.14]{Arumugam} and~\cite{LauNgShiu-chila}, for odd $m\ge 3, n\ge 1$, $\chi_{la}(C_m)=\chi(C_m)=3$ and $\chi_{la}(C_m\vee O_n)=4 =\chi(C_m)+1$.

\begin{problem} Show that the condition $n\ge m$ in Theorem~\ref{thm-G+On} or the condition $r\ge \frac{(m-n)(mn+1)}{2mn}$ in Theorem~\ref{thm-Gr+On} can be omitted. \end{problem}

\section{Affirmative answers to~\cite[Theorem 2.15 and Problem 3.3]{Arumugam}}

In~\cite[Theorem 2.15]{Arumugam}, the authors show that $3\le \chi_{la}(W_n) \le 5$ for $n\equiv 0\pmod{4}$. We now give the exact value of $\chi_{la}(W_n)$.

\begin{theorem}~\label{thm-W4k} For $k\ge 1$, $\chi_{la}(W_{4k}) = 3$. \end{theorem}

\begin{proof} Let $V(W_{4k}) = \{v\}\cup \{u_i : 1\le i\le 4k\}$ and $E(W_{4k})=\{vu_i: 1\le i\le 4k\} \cup \{u_iu_{i+1} : 1\le i\le 4k\}$, where $u_{4k+1}=u_1$. For $k=1$ and $2$, we have the labelings $f$ in figures below for $W_4$ and $W_8$ showing that $c(f)=3$.\\[3mm]

\centerline{
\raisebox{6mm}[1cm][1cm]{\epsfig{file=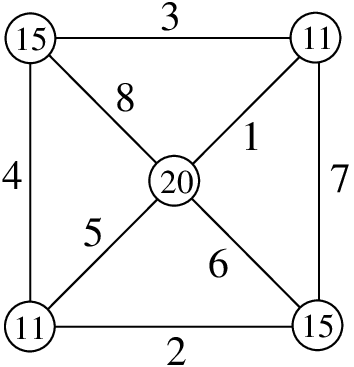, width=2.5cm}}\qquad\epsfig{file=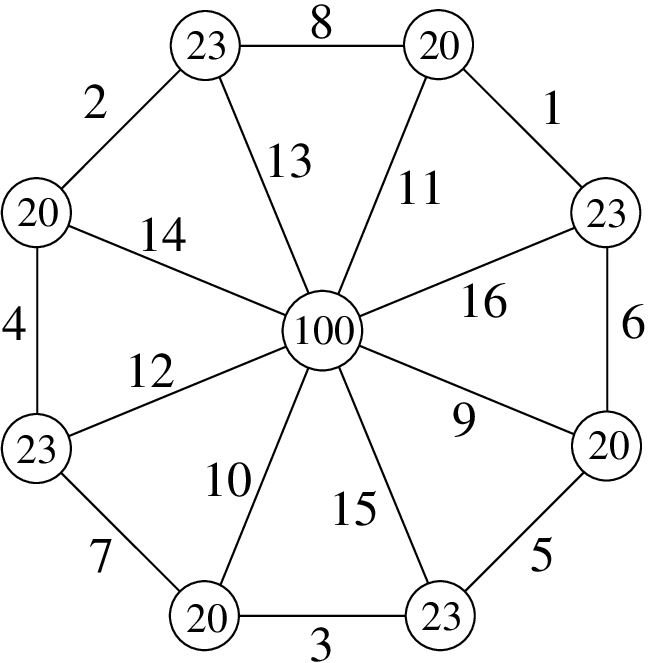, width=3.5cm}
}

For $k\ge 3$, we consider the following two tables.
%\vskip-5mm
\begin{table}[H]
\footnotesize
Table 1.\\
%      \centering
\setlength{\extrarowheight}{1pt}
\setlength{\tabcolsep}{1.5pt}
\vskip-0.3cm
\begin{tabular}{|c|c|c|c|c|c|c"c|c|c|c|c|c|c|}
\hline
$C_1$ & $C_2$ & $C_3$ & $\ldots$ & $C_{k-2}$ & $C_{k-1}$ & $C_{k}$ & $C_{k+1}$ & $C_{k+2}$ & $C_{k+3}$ & $\ldots$  & $C_{2k-2}$ & $C_{2k-1}$ & $C_{2k}$\\
\hline
1 & 3 & 5  & $\ldots$ & $2k - 5$ & $2k - 3$ & $2k - 1$ & 2 & 4 & 6 & $\ldots$ & $2k - 4$ & $2k - 2$ & $2k$\\
\hline
$3k$ & $3k - 1$ & $3k - 2$ & $\ldots$ & $2k + 3$ & $2k + 2$ & $2k + 1$ & $4k$ & $4k - 1$ & $4k - 2$ & $\ldots$ & $3k + 3$ & $3k + 2$ & $3k + 1$\\
\hline
$8k$ & $8k - 1$ & $8k - 2$ & $\ldots$ & $7k + 3$ & $7k + 2$ & $7k + 1$ & $7k - 1$ & $7k - 2$ & $7k - 3$ & $\ldots$ & $6k + 2$ & $6k + 1$ & $6k$\\
\hline
\end{tabular}
%     \centering
\end{table}
\vskip-0.5cm
\begin{table}[H]
\footnotesize
Table 2.\\
%      \centering
\setlength{\extrarowheight}{1pt}
\setlength{\tabcolsep}{1.2pt}
\vskip-0.3cm
\begin{tabular}{|c|c|c|c|c|c|c|c"c|c|c|c|c|c|}
\hline
$C_1$ & $C_2$ & $C_3$ & $\ldots$ & $C_{k-2}$ & $C_{k-1}$ & $C_{k}$ & $C_{k+1}$ & $C_{k+2}$ & $C_{k+3}$ & $C_{k+4}$ & $\ldots$ & $C_{2k-1}$ & $C_{2k}$\\
\hline
$3k + 2$ & $3k + 1$ & $3k$ & $\ldots$ & $2k + 5$ & $2k + 4$ & $2k + 3$ & $2k + 2$ & $2k$ & $4k$ & $4k - 1$ & $\ldots$ & $3k + 4$ & $3k + 3$\\
\hline
1 & 3 & 5 & $\ldots$ & $2k - 5$ & $2k - 3$ & $2k - 1$ & $2k + 1$ & 2 & 4 & 6 & $\ldots$ & $2k - 4$ & $2k - 2$\\
\hline
$6k - 1$ & $6k - 2$ & $6k - 3$ & $\ldots$ & $5k + 2$ & $5k + 1$ & $5k$ & $5k - 1$ & $7k$ & $5k - 2$ & $5k - 3$ & $\ldots$ & $4k + 2$ & $4k + 1$\\
\hline
\end{tabular}
%     \centering
\end{table}

Observe that
\begin{enumerate}[(i)]
  \item all integers in row 1 and row 2 of each table are in $[1, 4k]$;
  \item the two rows 3 of both tables collectively give all integers in $[4k + 1, 8k]$;
  \item Table 1 has constant column sum of $11k + 1$ and Table 2 has constant column sum of $9k + 2$;
  \item in Table 1, all integers from column $C_1$ to $C_k$, and from $C_{k+1}$ to $C_{2k}$ of each row form an arithmetic progression;
  \item in Table 2, all integers from column $C_1$ to $C_{k+1}$, and from $C_{k+2}$ to $C_{2k}$ of each row (or from $C_{k+3}$ to $C_{2k}$ for  row 1 and row 3) form an arithmetic progression.
\end{enumerate}

Consider the following three sequences obtained by taking the first two entries of a particular column of Table 1 and the first two entries of a particular column of Table 2 alternately. Both entries taken are written in ordered pair respectively.\\

For even $k$, we have
\begin{enumerate}[(a)]
  \item $(1, 3k)$, $(3k, 5)$, $(5, 3k - 2)$, $(3k - 2, 9)$, $(9, 3k - 4)$, $(3k - 4, 13)$, $\ldots$, $(2k - 7, 2k + 4)$, $(2k + 4, 2k - 3)$, $(2k - 3, 2k + 2)$, $(2k + 2, 2k + 1)$;
  \item $(2k + 1, 2k - 1)$, $(2k - 1, 2k + 3)$, $(2k + 3, 2k - 5)$, $(2k - 5, 2k + 5)$, $\ldots$, $(3k - 3, 7)$, $(7, 3k - 1)$, $(3k - 1, 3)$, $(3, 3k + 1)$;
  \item $(3k + 1, 2k)$, $(2k, 2)$, $(2, 4k)$, $(4k, 4)$, $(4, 4k - 1)$, $(4k - 1, 6)$, $(6, 4k - 2)$,
$(4k-2, 8)$, $\ldots$, $(2k-6, 3k+4)$, $(3k+4, 2k-4)$, $(2k-4, 3k+3)$, $(3k+3, 2k-2)$,
$(2k - 2, 3k + 2)$, $(3k + 2, 1)$.
\end{enumerate}

Here, sequences (a) and (b) are of length $k$ and sequence (c) is of length $2k$.

Observe that $T = (a) + (b) + (c)$ is a sequence of $4k$ ordered pairs with every integers in $[1, 4k]$ appearing exactly twice, once as the left entry of an ordered pair and once as the right entry of another ordered pair. Therefore, taking the left entry of every ordered pair gives us a sequence $S$ with $4k$ distinct integers in $[1, 4k]$. Define, $f : E(C_{4k}) \to S$ such that $f(e_i) = f(u_iu_{i+1})$ is the $i$-th entry of $S$, $1 \le i \le 4k$ and $u_{4k+1} = u_1$. Let $f(vu_{i+1})$ be the value in row 3 of the column that corresponds to the $i$-th entry of $S$. For  $1 \le j \le 2k$, since all the $(2j-1)$-st ordered pairs of $T$ are from Table 1 and all the $2j$-th ordered pairs are from Table 2, we now have $f^+(u_{2j}) = 11k + 1$ and $f^+(u_{2j-1}) = 9k + 2$. Moreover, $f^+(v) = (4k + 1) + \cdots + (8k) = 2k(12k + 1)$. Thus, $f$ is a local antimagic labeling of $W_{4k}$ with $c(f)=3$.

For odd $k$, we shall have different sequence (a) and sequence (b) as follows.
\begin{enumerate}[(a)]
  \item $(1, 3k)$, $(3k, 5)$, $(5, 3k - 2)$, $(3k - 2, 9)$, $(9, 3k - 4)$, $(3k - 4, 13)$, $\ldots$, $(2k - 5, 2k + 3)$, $(2k + 3, 2k - 1)$;
  \item $(2k - 1, 2k + 1)$, $(2k + 1, 2k + 2)$, $(2k + 2, 2k - 3)$, $(2k - 3, 2k + 4)$, $(2k + 4, 2k - 7)$, $(2k-7, 2k + 6)$, $\ldots$, $(3k - 3, 7)$, $(7, 3k - 1)$, $(3k - 1, 3)$, $(3, 3k + 1)$.
\end{enumerate}
 Here, sequences (a) and (b) are of length $k-1$ and $k+1$, respectively.

By an argument similar to that for even $k$, we also can obtain a local antimagic labeling $f$ of $W_{4k}$ with $c(f)=3$ such that $f^+(u_{2j}) = 11k + 1$ and $f^+(u_{2j-1}) = 9k + 2$ for  $1\le j\le 2k$, and $f^+(v) = 2k(12k + 1)$. Since $\chi_{la}(W_{4k}) \ge \chi(W_{4k}) = 3$, the theorem holds. 
\end{proof}

\begin{example}{\rm
For $k=3$, the tables defined in the proof of Theorem~\ref{thm-W4k} are\\%[-5mm]
\begin{minipage}{0.5\textwidth}
\begin{table}[H]
\footnotesize
Table 1.\\
\begin{tabular}{|c|c|c"c|c|c|}
\hline
$C_1$ & $C_2$ & $C_3$ & $C_4$ & $C_5$ & $C_6$\\
\hline
1 & 3 & 5 & 2  & 4 & 6\\
\hline
9 & 8 & 7 & 12 & 11 & 10\\
\hline
24 & 23 & 22 & 20 & 19 & 18\\
\hline
\end{tabular}
\end{table}
\end{minipage}
\begin{minipage}{0.5\textwidth}
\begin{table}[H]
\footnotesize
Table 2.\\\begin{tabular}{|c|c|c|c"c|c|}
\hline
$C_1$ & $C_2$ & $C_3$ & $C_4$& $C_5$ & $C_6$\\
\hline
11 & 10 & 9 & 8 & 6  & 12\\
\hline
1 & 3 & 5 & 7 & 2 & 4\\
\hline
17 & 16 & 15 & 14 & 21 & 13\\
\hline
\end{tabular}\end{table}
\end{minipage}

\vspace*{3mm}
Thus, sequence $T$ is given by:

\begin{enumerate}[(a)]
\item (1,9), (9,5);
\item (5,7), (7,8), (8,3), (3,10);
\item (10,6), (6,2), (2,12), (12,4), (4,11), (11,1)
\end{enumerate}
\vspace*{0mm}
and $S=\{1,9,5,7,8,3,10,6,2,12,4,11\}$.
Hence we have the following labeling:\\[1mm]
\centerline{
\epsfig{file=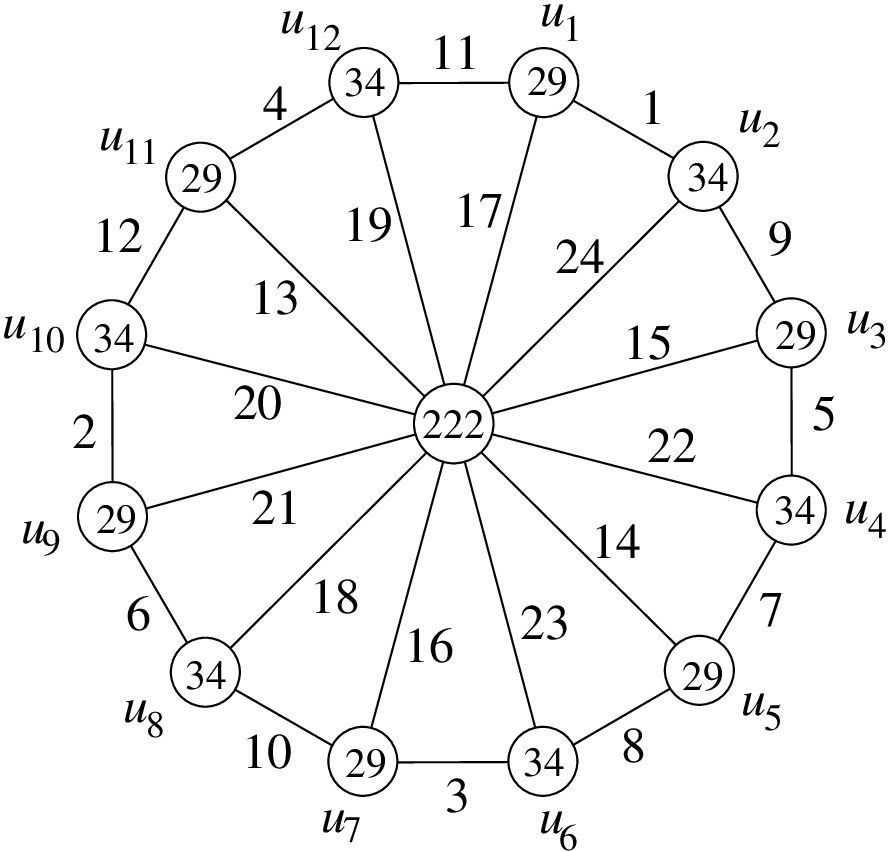, width=5cm}
}
}
\end{example}

In~\cite[Problem 3.3]{Arumugam}, the authors also asked:

\begin{problem}\label{pbm-nk} Does there exist a graph $G$ of order $n$ with $\chi_{la}(G) = n - k$ for every
$k = 0, 1, 2, \ldots, n - 2$? \end{problem}

The following theorem in~\cite{LauShiuNg} is needed to answer Problem~\ref{pbm-nk}. For completeness, the proof is also given.

\begin{theorem}\label{thm-pendant}  Let $G$ be a graph having $k$ pendants. If $G$ is not $K_2$, then $\chi_{la}(G)\ge k+1$ and the bound is sharp.\end{theorem}

\begin{proof} Suppose $G$ has size $m$. Let $f$ be any local antimagic labeling of $G$.
Consider the edge $uv$ with $f(uv) = m$. We may assume $u$ is not a pendant. Clearly, $f^+(u) > m \ge f^+(z)$ for every pendant $z$. Since all pendants have distinct induced colors, we have $\chi_{la}(G)\ge k+1$.\\

 For $k\ge 2$, since $\chi_{la}(S_k)= k+1$, where $S_k$ is a star with maximum degree $k$, the lower bound is sharp. The labeling in figure below shows that the lower bound is sharp for $k=1$.\\\\
\centerline{\epsfig{file=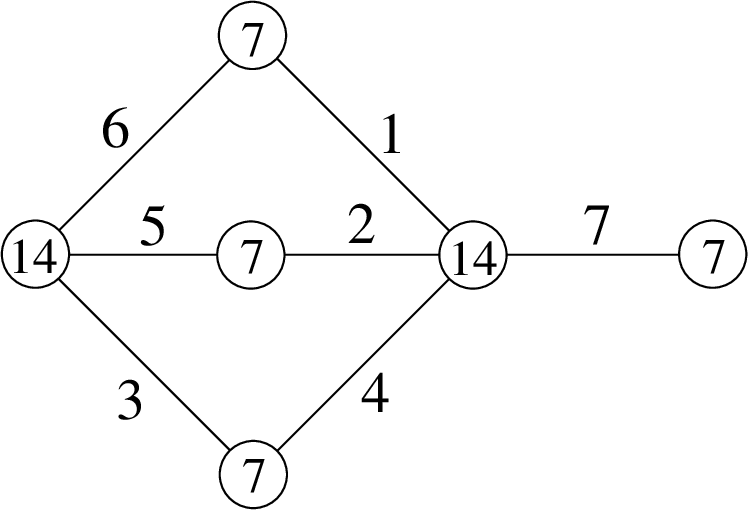, width=3.3cm}}\label{fig:theta61}
 \end{proof}

For $m\ge 2, t\ge 1$, let $CT(m,t)$ be the coconut tree  obtained by identifying the central vertex of a $K(1,t)$ with an end-vertex of a path $P_m$. Note that $CT(2,t)=K(1,t+1)$ with $\chi_{la}(K(1,t+1))=t+2$. Moreover, $CT(m,1) = P_{m+1}$.

\begin{theorem}\label{thm-coconut}  For $m\ge 2, t\ge 1$, $\chi_{la}(CT(m,t))=t+2$.  \end{theorem}

\begin{proof}  By~\cite[Theorem 2.7]{Arumugam}, the theorem holds for $t=1$. Consider $t\ge 2$. Let $P_m = v_1v_2\cdots v_{m}$ and $E(K(1,t)) = \{v_mx_j : 1\le j\le t\}$. Denote by $e_i$ the edge $v_iv_{i+1}$ for $1\le i\le m-1$. Define  $f:E(CT(m,t))\to [1,m+t-1]$ by
\begin{enumerate}[(i)]
  \item $f(e_i) = (i+1)/2$ for odd $i$,
  \item $f(e_i) = m - i/2$ for even $i$,
  \item $f(v_mx_j) = m+j-1$ for $1\le j\le t$.
\end{enumerate}

It is easy to verify that $f$ is a bijection with $f^+(x_j)=m+j-1$ for $1\le j\le t$, $f^+(v_m)\ge 2m+1$, $f^+(v_1)=1$, $f^+(v_i) = m+1$ for even $1 < i < m$ and $f^+(v_i)=m+2$ for odd $1 < i < m$. Thus, $f$ is a local antimagic labeling that induces a $(t+2)$-coloring so that $\chi_{la}(CT(m,t))\le t+2$. By Theorem~\ref{thm-pendant}, we know that $\chi_{la}(CT(m,t))\ge t+2$. Hence, $\chi_{la}(CT(m,t))=t+2$.   \end{proof}

%\begin{theorem} For each $n\ge k+2$, there exists a graph $G$ of order $n$ such that $\chi_{la}(G) = n - k$ if and only if $k=0,1,2,\ldots, n-3$. \end{theorem}

\begin{theorem}\label{thm-n-k} For each possible $n,k$, there exists a graph $G$ of order $n$ such that $\chi_{la}(G) = n - k$ if and only if $n\ge k+3\ge 3$. \end{theorem}

\begin{proof} By definition, $k\ge 0$ and $n\ge 3$. Suppose $n\ge k+3\ge 3$. Let $G=CT(m,t)$ of order $n=m+t\ge 3$. For $t\ge 1$, we have that $\chi_{la}(CT(m,t))=t+2=n-(m-2)$. Letting $m-2=k$, we have $\chi_{la}(CT(k+2,n-k-2))=n-k$. Thus, for every possible $0\le k\le n-3$, there is a graph $G$ of order $n$ such that $\chi_{la}(G)=n-k$. This proves the sufficiency.

We prove the necessity by contrapositive. Suffice to assume $n=k+2$. It is easy to check that there is no graph $G$ of order $n=3,4$ such that $\chi_{la}(G) = n-k=2$. 
\end{proof}

 Theorem~\ref{thm-n-k} and the following theorem in~\cite{LauShiuNg} solve Problem~\ref{pbm-nk} completely. For completeness, the proof is stated as well.

\begin{theorem} Suppose $n\ge 3$. There is a graph $G$ of order $n$ with $\chi_{la}(G)=2$ if and only if $n\ne 3, 4, 5, 7$.
\end{theorem}

\begin{proof} Suppose $n=3,4,5,7$, it is routine to check that all graphs $G$ of order $n$ has $\chi_{la}(G)\ge 3$. This proves the necessity by contrapositive.

We now prove the sufficiency. Suppose $n$ is odd and $n\ge 9$. Since $n=6s+1$ ($s\ge 2$), $n=6s+3$ ($s\ge 1$) or $6s+5$ ($s\ge 1$), we consider the following three cases.

 {\bf Case (a).} $n=6s+1$. Suppose $s\ge 3$. We shall construct a bipartite graph $G$ with bipartition $(A,B)$, where $|A|=3$ and $|B|=6s-2$, such that all vertices in $B$ are of degree $2$. If $G$ exists, then $G$ is of order $6s+1$ and size $12s-4$. Suppose there is a local antimagic labeling $f$ of $G$ such that $c(f)=2$, then $f$ corresponds to a labeling matrix $M$
%
%\[\begin{pmatrix}\bigstar & A\\
%A^T &\bigstar \end{pmatrix},\] where $A$
%
of size $3\times (6s-2)$ such that each of its entry is either an integer in $[1, 12s-4]$ or $*$. Moreover, each integer in $[1, 12s-4]$ appears as entry of $A$ once. Note that the total sum of integers in $[1, 12s-4]$ is $3(6s-2)(4s-1)$. We now arrange integers in $[1, 12s-4]$ to form matrix $M$ as follows: % such that each row sum is $(6s-2)(4s-1)$ and each column sum is $3(4s-1)$.

\begin{enumerate}[(1).]
  \item In row 1, assign $k$ to column $k$ if $k=2,4,6,\ldots,4s-2, 6s-2$; assign $12s-3-k$ to column $k$ if $k=3,5,7,\ldots, 4s-1$.
  \item In row 2, assign $k$ to column $k$ if $k=1,2s-1,2s+1,2s+3,\ldots,6s-3$; assign $12s-3-k$ to column $k$ if $k=2s, 2s+2, 2s+4, \ldots, 6s-4$.
  \item In row 3, assign $12s-4$ to column 1; assign $k$ to column $k$ if $k=3, 5, 7, \ldots, 2s-3, 4s, 4s+2, 4s+4, \ldots, 6s-4$; assign $12s-3-k$ to column $k$ if $k=2,4,6,\ldots, 2s-2, 4s+1, 4s+3, 4s+5, \ldots, 6s-3, 6s-2$.
  \item All the remaining columns of each row is assigned with $*$.
\end{enumerate}

 The resulting matrix is given by the following table:\\
$\begin{array}{|c"c|c|c|c"c|c|c|c|}\hline
* & 2 & 12s-6 & \cdots & 2s-2 & 10s-2 & 2s & 10s-4 & \cdots \\\hline
1 & * & * & \cdots & * & 2s-1 & 10s-3 &  2s+1  & \cdots \\\hline
12s-4 & 12s-5 & 3 & \cdots & 10s-1 &  * & * & * & *  \\\hline
\end{array}$\\[3mm]
$\begin{array}{|c|c|c"c|c|c|c|c"c|}\hline
 8s & 4s-2 & 8s-2 & * & \dots & * & * & * & 6s-2\\\hline
 4s-3 & 8s-1 & 4s-1 & 8s-3 & \cdots & 6s-5 & 6s+1 & 6s-3 & *\\\hline
 * & * & * & 4s & \cdots & 6s+2 & 6s-4 & 6s & 6s-1\\\hline
\end{array}$\\

 It is easy to check that the first row contains $4s-1$ numbers, the second row contains $4s$ numbers and the third row contains $4s-3$ numbers. Moreover, each column sum is $12s-3$ and each row sum is $(6s-2)(4s-1)$. Thus, $G$ exists and $\chi_{la}(G)=2$.

When $s=2$ $(n=13)$, a required labeling matrix is as follow:
\begin{center} $M=\begin{array}{|c|c|c|c|c|c|c|c|c|c|}\hline
* & 2 & 18 & 4 & 16 & 6 & 14 & * & * & 10\\\hline
1 & * & 3 & 17 & 5 & 15 & 7 & 13 & 9 & *\\\hline
20 & 19 & * & * & * & * & * & 8 & 12 & 11 \\\hline
\end{array}$\end{center}

{\bf Case (b).} $n=6s+3$. Similar to Case (a), a labeling matrix $M$ of size $3\times 6s$ such that each of its entry is either an integer in $[1, 12s]$ or $*$ can be obtained.

For $s\ge 1$, we let the matrix $M=\begin{pmatrix}M_1 &M_2 &M_3\end{pmatrix}$, where each $M_i$ is a $3\times 2s$ matrix.
For the first matrix, we assign $1,2,\dots, 2s$ at the first row; $*$ at each entry of the second row; $12s, 12s-1, \dots, 10s+1$ at the third row. We then swap the $(1, j)$-entry with $(3, j)$-entry of this matrix when $j\equiv 2,3\pmod{4}$ and $2\le j\le 2s$. The resulting matrix is $M_1$. Similarly, for the second matrix, we assign $10s, 10s-1, \dots, 8s+1$ at the first row; $2s+1,2s+2,\dots, 4s$ at the second row; $*$ at each entry of the third row. We then swap the $(1, j)$-entry with $(2, j)$-entry of this matrix when $j\equiv 2,3\pmod{4}$ and $2\le j\le 2s$. The resulting matrix is $M_2$. For the third matrix, we assign $*$ at each entry of the first row;  $8s, 8s-1, \dots, 6s+1$ at the second row; $4s+1,4s+2,\dots, 6s$ at the third row. We then swap the $(2, j)$-entry with $(3, j)$-entry of this matrix when $j\equiv 2,3\pmod{4}$ and $2\le j\le 2s$. The resulting matrix is $M_3$.

So when $s$ is odd, we have\\[1mm]
\setlength{\tabcolsep}{0.6pt}
$M_1=\small\begin{array}{|c|c"c|c|c|c"c"c|c|c|c|}\hline
1 & 12s-1 & 12s-2 & 4  & 5 & 12s-5 & \cdots &  10s+4 & 2s-2 & 2s-1 & 10s+1 \\\hline
* & * & * & * & * & * & \cdots &  * & * & * & *  \\\hline
12s & 2 & 3 & 12s-3 & 12s-4 & 6& \cdots &  2s-3 & 10s+3 & 10s+2 & 2s \\\hline
\end{array}$\\[3mm]
$\setlength{\tabcolsep}{0.8pt}M_2=\small\begin{array}{|c|c"c|c|c|c"c"c|c|c|c|}\hline
10s & 2s+2 & 2s+3 & 10s-3 & 10s-4 & 2s+6 & \cdots &  4s-3 & 8s+3 & 8s+2 & 4s \\\hline
2s+1 & 10s-1 & 10s-2 & 2s+4 & 2s+5 & 10s-5 & \cdots &  8s+4 & 4s-2 & 4s-1 & 8s+1 \\\hline
* & * & * & * &  * & * & \cdots & * & * & * & *  \\\hline
\end{array}$\\[3mm]
$\setlength{\tabcolsep}{0.8pt}M_3=\small\begin{array}{|c|c"c|c|c|c"c"c|c|c|c|}\hline
* & * & * & * &  * & * & \cdots & * & * & * & *  \\\hline
8s & 4s+2 & 4s+3 & 8s-3 & 8s-4 & 4s+6 &\cdots &  6s-3 & 6s+3 & 6s+2 & 6s \\\hline
4s+1 & 8s-1 & 8s-2 & 4s+4 & 4s+5 & 8s-5 & \cdots &  6s+4 & 6s-2 & 6s-1 & 6s+1 \\\hline
\end{array}$\\[1mm]

One may check that the first row sum of $M_1$, the second row sum of $M_2$ and the third row sum of $M_3$ are the same which is $12s+\frac{1}{4}(2s-2)(24s+2)=12s^2+s-1$. It is easy to see that the third row sum of $M_1$, the first row sum of $M_2$ and the second row sum of $M_3$ also are the same and equals $(12s+1)(2s)-(12s^2+s-1)=12s^2+s+1$. Hence each row sum of $M$ is $24s^2+2s$ and each column sum is $12s+1$.

%Suppose $s$ is even. $M_1$, $M_2$ and $M_3$ are obtained by the same initial arrangement and swapping as odd $s$ case.
%For the first matrix, we assign $1,2,\dots, 2s$ at the first row; $*$ at each entry of the second row; $12s, 12s-1, \dots, 10s+1$ at the third row. We then swap the $(1, j)$-entry with $(3, j)$-entry of this matrix when $j\equiv 2,3\pmod{4}$ and $2\le j\le 2s$. The resulting matrix is $M_1$. Similarly, for the second matrix, we assign $2s+1,2s+2,\dots, 4s$ at the first row; $10s, 10s-1, \dots, 8s+1$ at the second row; $*$ at each entry of the third row. We then swap the $(1, j)$-entry with $(2, j)$-entry of this matrix when $j\equiv 2,3\pmod{4}$ and $2\le j\le 2s$. The resulting matrix is $M_2$. For the third matrix, we assign $*$ at each entry of the first row; $4s+1,4s+2,\dots, 6s$ at the second row;  $8s, 8s-1, \dots, 6s+1$ at the third row. We then swap the $(2, j)$-entry with $(3, j)$-entry of this matrix when $j\equiv 2,3\pmod{4}$ and $2\le j\le 2s$. The resulting matrix is $M_3$.
When $s$ is even, we have\\[1mm]
$\setlength{\tabcolsep}{0.8pt}M_1=\small\begin{array}{|c|c|c|c"c|c|c"c|c|c|c|}\hline
1 & 12s-1 & 12s-2 & 4 & 5 & \cdots & 2s-4 &2s-3 & 10s+3 & 10s+2 & 2s \\\hline
* & * & * & * & * & \cdots & * & * & * & * & *\\\hline
12s & 2 & 3 & 12s-3 & 12s-4 &\cdots & 10s+5& 10s+4 & 2s-2 & 2s-1 & 10s+1 \\\hline
\end{array}$\\[3mm]
$\setlength{\tabcolsep}{0.8pt}M_2=\small\begin{array}{|c|c|c|c"c|c|c"c|c|c|c|}\hline
10s & 2s+2 & 2s+3 & 10s-3 & 10s-4 & \cdots &  8s+5& 8s+4 & 4s-2 & 4s-1 & 8s+1 \\\hline
2s+1 & 10s-1 & 10s-2 & 2s+4 & 2s+5 & \cdots & 4s-4 & 4s-3 & 8s+3 & 8s+2 & 4s \\\hline
* & * & * & * & * & \cdots & * & * &  * & * & * \\\hline
\end{array}$\\[3mm]
$M_3=\small\begin{array}{|c|c|c|c"c|c|c"c|c|c|c|}\hline
* & * & * & * & * & \cdots  & * & * & * & * & * \\\hline
8s & 4s+2 & 4s+3 & 8s-3 & 8s-4 & \cdots & 6s+5 & 6s+4 & 6s-2 & 6s-1 & 6s+1 \\\hline
4s+1 & 8s-1 & 8s-2 & 4s+4 & 4s+5 & \cdots & 6s-4& 6s-3 & 6s+3 & 6s+2 & 6s \\\hline
\end{array}$\\[1mm]

One may check that all the numerical row sum of $M_i$ are the same which is $\frac{1}{4}(2s)(24s+2)+10s+2=12s^2+s$. Hence each row sum of $M$ is $24s^2+2s$ and each column sum is $12s+1$.

Thus, $G$ exists and $\chi_{la}(G)=2$.

{\bf Case (c).} $n=6s+5$. Suppose $s\ge 2$. We shall construct a bipartite graph $G$ with bipartition $(A,B)$, where $|A|=3$ and $|B|=6s+2$, such that $B$ has a vertex of degree 1 and the remaining $6s+1$ vertices are of degree $2$. If $G$ exists, then $G$ is of order $6s+5$ and size $12s+3$. Note that the total sum of integers in $[1, 12s+3]$ is $3(6s+2)(4s+1)$. Similar to the above construction, we want to arrange integers in $[1, 12s+3]$ to form a $3\times (6s+2)$ matrix $M$ as follows:

\begin{enumerate}[(1).]
  \item In row 1, assign $k$ to column $k$ if $k=2,4,6,\ldots,4s+2, 6s$; assign $12s+3-k$ to column $k$ if $k=3,5,7,\ldots, 4s+1$.
  \item In row 2, assign $k$ to column $k$ if $k=1,2s+1,2s+3,2s+5,\ldots,6s+1$; assign $12s+3-k$ to column $k$ if $k=2s, 2s+2, 2s+4, \ldots, 6s-2$.
  \item In row 3, assign $12s+2$ to column 1 and $12s+3$ to column $6s+2$; assign $k$ to column $k$ if $k=3, 5, 7, \ldots, 2s-1, 4s, 4s+2, 4s+4, \ldots, 6s-2$; assign $12s+3-k$ to column $k$ if $k=2,4,6,\ldots, 2s-2, 4s+3, 4s+5, 4s+7, \ldots, 6s+1$.
  \item All the remaining columns of each row is assigned with $*$.
\end{enumerate}
The resulting matrix is given by the following table:\\
$\begin{array}{|c|c"c|c|c|c"c|c|c|c|}\hline
* & 2 & 12s & 4 & \cdots & 10s+4 & 2s & 10s+2 & \cdots & 4s \\\hline
1 & * & * & * & \cdots & * & 10s+3 & 2s+1 & \cdots &  8s+3 \\\hline
12s+2 & 12s+1 & 3 & 12s-1 & \cdots & 2s-1 & * & * &\cdots & *\\\hline
\end{array}$\\[3mm]
$\begin{array}{|c|c"c|c|c|c"c|c"c|}\hline
8s+2 & 4s+2 & * & \cdots & * & * & 6s & * & *\\\hline
4s+1 & 8s+1 & 4s+3 & \cdots &  6s+5 & 6s-1 & *  & 6s+1 & *\\\hline
* & * & 8s & \cdots & 6s-2 & 6s+4 & 6s+3  & 6s+2 & 12s+3\\\hline
\end{array}$

 It is easy to check that the first row contains $4s+1$ numbers, the second row contains $4s+2$ numbers and the third row contains $4s$ numbers. Moreover, each column sum is $12s+3$ and each row sum is $(6s+2)(4s+1)$.

When $s=1$ $(n=11)$, a required labeling matrix is as follow:
\begin{center} $M=\begin{array}{|*{8}{c|}}\hline
1 & 3 & 4 & 5 & 6 & 8 & 13 & *\\\hline
* & 12 & * & 10 & 9 & 7 & 2 & * \\\hline
14 & * & 11 & * & * & * & * & 15\\\hline
\end{array}$
\end{center}

 Thus, $G$ exists and $\chi_{la}(G)=2$. \\

Suppose $n\ge 6$ is even. In~\cite[Theorem 2.11]{Arumugam} (see Theorem~\ref{thm-Kpq}), we have $\chi_{la}(K_{p,q})=2$, where $p\ne q$ and $n=p+q\ge 6$ is an even integer. Therefore, there exists a graph $G$ of order $n$ such that $\chi_{la}(G)=2$ for every even $n\ge 6$.
 \end{proof}

\begin{example} Let $n=19$, we have
\begin{center} $M=\begin{array}{|c"c|c|c"c|c|c|c|c|c|c"c|c|c|c"c|}\hline
* & 2 & 30 & 4 & 28 & 6 & 26 & 8 & 24 & 10 & 22 & * & * & * & * & 16\\\hline
1 & * & * & * & 5 & 27 & 7 & 25 & 9 & 23 & 11 & 21 & 13 & 19 & 15 & *\\\hline
32 & 31 & 3 & 29 & * & * & * & * & * & * & * & 12 & 20 & 14 & 18 & 17 \\\hline
\end{array}$\end{center}

Let $n=21$, we have
\begin{center} $M=\begin{array}{|c|c"c|c|c|c||c|c"c|c|c|c||c|c"c|c|c|c|}\hline
1 & 35 & 34 & 4 & 5 & 31 & 30 & 8 & 9 & 27 & 26 & 12 & * & * & * & * & * & *\\\hline
* & * & * & * & * & * & 7 & 29 & 28 & 10 & 11 & 25 & 24 & 14 & 15 & 21 & 20 & 18\\\hline
36 & 2 & 3 & 33 & 32 & 6 & * & * & * & * & * & * & 13 & 23 & 22 & 16 & 17 & 19\\\hline
\end{array}$\end{center}

Let $n=15$, we have
\begin{center} $M=\begin{array}{|c|c|c|c||c|c|c|c||c|c|c|c|}\hline
1 & 23 & 22 & 4 & 20 & 6 & 7 & 17 & * & * & * & *\\\hline
* & * & * & * & 5 & 19 & 18 & 8 & 16 & 10 & 11 & 13\\\hline
24 & 2 & 3 & 21 & * & * & * & * & 9 & 15 & 14 & 12\\\hline
\end{array}$\end{center}

Let $n=17$, we have
\begin{center} $M=\begin{array}{|c|c"c|*{6}{c|}c"c"c|c"c|}\hline
* & 2 & 24 & 4 & 22 & 6 & 20 & 8 & 18 & 10 & * & 12 & * & *\\\hline
1 & * & * & 23 & 5 & 21 & 7 & 19 & 9 & 17 & 11 & * & 13 & * \\\hline
26 & 25 & 3 & * & * & * & * & * & * & *  & 16 & 15 & 14 & 27 \\\hline
\end{array}$
\end{center}
\end{example}

\section{Complete Bipartite Graphs}

\begin{theorem}\label{thm-Kpq}\cite[Theorems 2.11-12]{Arumugam} For $p,q\ge 1$ and $(q,p)\ne (1,1)$, $$\chi_{la}(K_{p,q})=\begin{cases} q+1 &\mbox{ if } q > p = 1,\\
3 &\mbox{ if } p=2, q=2\mbox{ or } q \mbox{ is odd},\\
 2 &\mbox{ if } p\ge 2, p\not= q \mbox{ and } p\equiv q\pmod{2}.\end{cases}$$ \end{theorem}

We next determine $\chi_{la}(K_{p,q})$ for all $p,q$ not considered in Theorem~\ref{thm-Kpq}. Suppose $f$ is a local antimagic labeling of $K_{p,q}$ and $M$ is a $p\times q$ matrix with row sums and column sums correspond to the vertex labels under $f$ accordingly.

The following lemma~\cite{LauNgShiu-chila} is needed.
\begin{lemma}\label{lem-2part} Let $G$ be a graph of size $q$. Suppose there is a local antimagic labeling of $G$ inducing a $2$-coloring of $G$ with colors $x$ and $y$, where $x<y$. Let $X$ and $Y$ be the numbers of vertices of colors $x$ and $y$, respectively. Then $G$ is a bipartite graph whose sizes of parts are $X$ and $Y$ with $X>Y$, and $xX=yY= \frac{q(q+1)}{2}$. \end{lemma}

\begin{theorem}\label{thm-Kpp} For $p\ge 3$, $\chi_{la}(K_{p,p}) = 3$.
\end{theorem}

\begin{proof} By Lemma~\ref{lem-2part}, $\chi_{la}(K_{p,p})\ge 3$. Suppose $p\ge 4$ is even. Let $A$ be a magic rectangle of size $p\times 2$ by using integers in $[1, 2p]$ and $B$ be a magic rectangle of size $p\times (p-2)$ by using integers in $[2p+1, p^2]$. Note that, for the construction of magic rectangles, one may find from \cite{Chai, Reyes}. Now, for $M=\begin{pmatrix}A& B\end{pmatrix}$, each row sum is $\frac{1}{2}p(p^2+1)$, each of the first two column sums is $\frac{1}{2}p(2p+1)$, and each other column sum is $\frac{1}{2}p(p+1)^2$. Thus, $\chi_{la}(K_{p,p}) = 3$ for even $p\ge 2$.

Suppose $p=2n+1\ge 3$ is odd. Consider the $(2n+1)\times (2n+1)$ magic square $A$ constructed by Siamese method:\\
 Starting from the $(1, n+1)$-entry (i.e, $A_{1, n+1}$) with the number 1, the fundamental movement for filling the entries is diagonally up and right, one step at a time. When a move would leave the matrix, it is wrapped around to the last row or first column, respectively.
If a filled entry is encountered, one moves vertically down one box instead, then continuing as before. One may find the detail in \cite{Kraitchik}.

Note that each of the ranges $[1,p]$, $[p+1,2p]$, $\dots$, $[p^2-p+1,p^2]$ occupies a diagonal of the matrix, wrapping at the edges.
Namely, the range $[1,p]$ starts at $A_{1, n+1}$ and ends at $A_{2, n}$; the range $[p+1,2p]$ starts at $A_{3, n}$ ends at $A_{4, n-1}$; the range $[2p+1,3p]$ starts at $A_{5, n-1}$ and ends at $A_{6, n-2}$, etc. In general, the range $[ip+1,(i+1)p]$ starts at $A_{2i+1, n+1-i}$ and ends at $A_{2i+2, n-i}$, where $0\le i\le p-1$ and the indices are taken modulo $p$. It is easy to see that the $(n+1)$-st column of $A$ is $(1, p+2, \dots, p^2)$ which is an arithmetic sequence with common difference $p+1$.

Now let $M$ be the matrix obtained from $A$ by shifting up the $(n+1)$-st column by one entry (the top entry moves to the bottom). Hence each column sum of $M$ is still the magic number $\frac{1}{2}p(p^2+1)$. Each row sum of $M$ is $\frac{1}{2}p(p^2+1)+p+1$ except the last row sum which is $\frac{1}{2}p(p^2+1)-p^2+1$. Thus, we conclude that $\chi_{la}(K_{p,p})=3$.
\end{proof}

\begin{example} Suppose $p=5$. We have the following magic square of order $5$:

\centerline{$A=\begin{pmatrix}
17 & 24 & 1 & 8 & 15\\
23 & 5 & 7 & 14 & 16\\
4 & 6 & 13 & 20 & 22\\
10 & 12 & 19 & 21 & 3\\
11 & 18 & 25 & 2 & 9
\end{pmatrix}$}
Now

\centerline{$M=\begin{array}{|ccccc|c}
17 & 24 & 7 & 8 & 15 & 71\\
23 & 5 & 13 & 14 & 16 & 71\\
4 & 6 & 19 & 20 & 22 & 71\\
10 & 12 & 25 & 21 & 3 & 71\\
11 & 18 & 1 & 2 & 9 & 41\\\hline
65 & 65 & 65 & 65 & 65 & \mbox{sum}
\end{array}$}

\end{example}

\begin{theorem} For $p,q\ge 3$ and $p\not\equiv q\pmod{2}$, $\chi_{la}(K_{p,q}) = 3$.
\end{theorem}

\begin{proof} If $\chi_{la}(K_{p,q})=2$, then the corresponding matrix $M$ is a magic rectangle. Since there is no magic rectangle of size $p\times q$ for $p\not\equiv q\pmod{2}$, we know $\chi_{la}(K_{p,q})\ge 3$. Without loss of generality, assume $q\ge 3$ is odd and $p\ge 4$ is even.

Consider $p\ge 6$. Let $A$ be a $3\times q$ magic rectangle using integer in $[1,3q]$ and $B$ be a $(p-3)\times q$ magic rectangle  using integers in $[3q+1,qp]$. Let $M=\begin{pmatrix} A \\ B\end{pmatrix}$. Thus, each column sum of $M$ is $y=\frac{1}{2}(qp+1)p$, each of the first three row sums of $M$ is $x=\frac{1}{2}(3q+1)q$, and each other row sum of $M$ is $z=\frac{1}{2}(qp+3q+1)q$. Clearly $x<z$.

Suppose $p> q$. It is easy to see that $x<y$.
Consider $2(y-z)=(qp+1)(p-q)-3q^2$. If $p\ge q+3$, then $y-z>0$. Suffice to consider $p=q+1$. In this case, $2(y-z)=-2q^2+q+1\ne 0$ when $q\ge 3$. So $M$ corresponds a local antimagic labeling of $K_{p,q}$ for this case.

Suppose $q> p$. $2(x-y)=3q^2+(1-p^2)q-p\equiv -p\pmod{q}$. So $x-y\ne 0$ since $q>p\ge 6$. Now  $2(y-z)=(qp+1)(p-q)-3q^2<0$. So $M$ corresponds a local antimagic labeling of $K_{p,q}$ for this case.

The remaining case is when $p=4$. Let $A$ be a $4\times q$ matrix whose first row is the sequence of odd integers in $[1, 2q]$ in natural order; second row is the sequence of even integers in $[2q+1, 4q]$ in reverse natural order; third row is the sequence of even integers in $[1, 2q]$ in natural order; last row is the sequence of odd integers in $[2q+1, 4q]$ in reverse natural order. It is clear that each column sum is $2(4q+1)$, the first row sum is $q^2$, the second row sum is $3q^2+q$, the third row sum is $q^2+q$, and the last row sum is $3q^2$.

Suppose $q\equiv 3\pmod 4$.
Now  $A_{1,(3q+3)/4}$ (the $(1,(3q+3)/4)$-entry of $A$) and $A_{2,(3q+3)/4}$ are $\begin{pmatrix}(3q+1)/2\\ (5q+1)/2\end{pmatrix}$, respectively. Swap these two entries to obtain a matrix $M$. Thus, the first row sum of $M$ is $q^2+q$, the second row sum is $3q^2$, the third row sum is $q^2+q$, and the last row sum is $3q^2$.

Suppose $q\equiv 1\pmod 4$.
Now the $A_{1,(3q+5)/4}$ and $A_{2,(3q+5)/2}$ are $\begin{pmatrix}(3q+3)/2\\ (5q-1)/2\end{pmatrix}$, respectively. To obtain $M$ we swap these two entries first, and then swap $A_{1,1}$ with $A_{3,1}$ and swap $A_{2,1}$ with $A_{4,1}$. Now, the first row sum of $M$ is $q^2+q-1$, the second row sum is $3q^2+1$, the third row sum is $q^2+q-1$, and the last row sum is $3q^2+1$.

It is easy to see that the column sum $8q+2$ cannot be equal to each row sum. Hence this completes the proof. \end{proof}

\begin{example} Consider the graph $K_{4,7}$.
Let
\[A=\begin{array}{|*{7}{c}|c}
1 & 3 & 5 & 7 & 9 & 11 & 13 & 49\\
28 & 26 & 24 & 22 & 20 & 18 & 16 & 154\\
2 & 4 & 6 & 8 & 10 & 12 & 14 & 56\\
27 & 25 & 23 & 21 & 19 & 17 & 15 & 147\\\hline
58 & 58 & 58 & 58 & 58 & 58 & 58 & \mbox{Sum}\end{array}\]
After swapping $A_{1, 6}$ with $A_{2,6}$ we have
\[M=\begin{array}{|*{7}{c}|c}
1 & 3 & 5 & 7 & 9 & 18 & 13 & 56\\
28 & 26 & 24 & 22 & 20 & 11 & 16 & 147\\
2 & 4 & 6 & 8 & 10 & 12 & 14 & 56\\
27 & 25 & 23 & 21 & 19 & 17 & 15 & 147\\\hline
58 & 58 & 58 & 58 & 58 & 58 & 58 & \mbox{Sum}\end{array}\]

Next we consider the graph $K_{4,5}$.
Let
\[A=\begin{array}{|*{5}{c}|c}
1 & 3 & 5 & 7 & 9 & 25\\
20 & 18 & 16 & 14 & 12 & 80\\
2 & 4 & 6 & 8 & 10& 30\\
19 & 17 & 15 & 13 & 11 & 75\\\hline
42 & 42 & 42 & 42 & 42  & \mbox{Sum}\end{array}\]
After swapping the $A_{1, 5}$ with $A_{2,5}$, $A_{1, 1}$ with $A_{3,1}$, and $A_{2, 1}$ with $A_{4,1}$ we have
\[M=\begin{array}{|*{5}{c}|c}
2 & 3 & 5 & 7 & 12 & 29\\
19 & 18 & 16 & 14 & 9 & 76\\
1 & 4 & 6 & 8 & 10& 29\\
20 & 17 & 15 & 13 & 11 & 76\\\hline
42 & 42 & 42 & 42 & 42  & \mbox{Sum}\end{array}\]
\end{example}

\begin{corollary} For $q\ge p\ge 1$ and $q\ge 2$,
$$\chi_{la}(K_{p,q})=\begin{cases} q+1 &\mbox{ if } q > p = 1,\\
 2 &\mbox{ if } q>p\ge 2 \mbox{ and } p\equiv q\pmod{2},\\
 3 &\mbox{ otherwise}.\end{cases}$$

%For $p,q\ge 1$ and $(q,p)\ne (1,1)$, $$\fontsize{11}{12}\selectfont\chi_{la}(K_{p,q})=\begin{cases} q+1 &\mbox{ if } q > p = 1,\\
%3 &\mbox{ if } p=q \mbox{ or } p\not\equiv q\pmod{2} \mbox{ is odd},\\
% 2 & otherwise.\end{cases}$$
\end{corollary}

\noindent {\bf Acknowledgements}  This paper was initiated during the first author's visit to Harbin Engineering University.  He is grateful to the university for providing full financial support and to Universiti Technologi MARA for granting the leaves. He would also like to dedicate this paper to Prof Yee-Hock Peng on the occasion of his 65th birthday.


\end{document}